\documentclass[12pt]{amsart}
\usepackage[dvips]{graphicx}
\usepackage{amsfonts, amsmath, amssymb, amscd, latexsym, pb-diagram}

\newcommand{\syms}{\mathsf}
\newcommand{\h}{\ensuremath{\mathcal{H}}}
\newcommand{\hn}{\ensuremath{\mathcal{H}_n}}
\newcommand{\hnp}{\ensuremath{\mathcal{H}_{np}}}
\newcommand{\sym}{\ensuremath{\syms{Sym}}}
\newcommand{\fin}{\ensuremath{\syms{FSym}}}
\newcommand{\alt}{\ensuremath{\syms{FAlt}}}

\newcommand\sesn[5]{1\longrightarrow #1\stackrel#2\longrightarrow #3\stackrel#4\longrightarrow #5\longrightarrow 1}

\def\N{{\mathbb N}}           
\def\Z{{\mathbb Z}}
\def\H{{\mathcal H}}              

\DeclareMathOperator{\aut}{Aut}
\DeclareMathOperator{\com}{Com}
\DeclareMathOperator{\qi}{QI}

\newtheorem{thm}{Theorem}[section]    
\newtheorem{lemma}[thm]{Lemma}          
\newtheorem{prop}[thm]{Proposition}%

\title[Metric Properties of Houghton's Groups]{Commensurations and Metric Properties of Houghton's groups}
\author[Burillo]{Jos\'e Burillo}
\address{Departament de Matem\`atica Aplicada IV\\
EETAC -- UPC\\
Esteve Terrades 5\\
08860 Castelldefels (Barcelona)\\Spain}
\email{burillo@ma4.upc.edu}

\author[Cleary]{Sean Cleary}
\address{Department of Mathematics R8133\\
The City College of New York\\\newline
Convent Ave \& 138th\\
New York, NY 10031\\USA}
\email{scleary@ccny.cuny.edu}

\author[Martino]{Armando Martino}
\address{Mathematics\\
University of Southampton\\
Highfield\\
Southampton\\
SO17 1BJ\\ England}
\email{A.Martino@soton.ac.uk}

\author[R\"over]{Claas E. R\"over}
\address{Department of Mathematics\\
National University of Ireland, Galway\\
University Road\\
Galway\\
Ireland}
\email{claas.roever@nuigalway.ie}

\thanks{The authors would like to thank Charles Cox for his remarks in an earlier version of this paper. The authors are also grateful for the hospitality of NUI Galway.
  The first author acknowledges support from MEC grant
  \#MTM2011-25955.  The second author acknowledges support from
  the National Science Foundation and that this work was partially
  supported by a grant from the Simons Foundation (\#234548 to Sean
  Cleary).}

\date{\today}

\begin{document}

\begin{abstract}\noindent
We describe the automorphism groups and the abstract commensurators of
Houghton's groups. Then we give sharp estimates for the word metric of
these groups and deduce that the commensurators embed into the
corresponding quasi-isometry groups. As a further consequence, we
obtain that the Houghton group on two rays is at least quadratically
distorted in those with three or more rays.
\end{abstract}

\maketitle

\section*{Introduction}

The family of Houghton groups \hn\ was introduced by Houghton
\cite{houghton}.  These groups serve as an interesting family of
groups, studied by Brown \cite{brownfp}, who described their
homological finiteness properties, by R\"over \cite{claasv}, who
showed that these groups are all subgroups of Thompson's group $V$,
and by Lehnert \cite{jorgdead} who described the metric for $\H_2$.
Lee \cite{lee} described isoperimetric bounds, and de Cornulier, Guyot, and Pitsch \cite{isolated}
showed that they are isolated points in the space of groups.

Here, we classify automorphisms and determine the abstract
commensurator of \hn.  We also give sharp estimates for the word
metric which are sufficient to show that the map from the abstract
commensurator to the group of quasi-isometries of \hn\ is an
injection.

\section{Definitions and background}\label{sec_bg}

Let $\N$ be the set of natural numbers (positive integers) and $n\ge
1$ be an integer. We write $\Z_n$ for the integers modulo $n$ with
addition and put $R_n=\Z_n\times\N$. We interpret $R_n$ as the graph
of $n$ pairwise disjoint rays; each vertex $(i,k)$ is
connected to $(i,k+1)$. We denote by $\sym_n$, $\fin_n$ and $\alt_n$,
or simply \sym, \fin\ and \alt\ if $n$ is understood, the full
symmetric group, the finitary symmetric group and the finitary
alternating group on the set $R_n$, respectively.

The {\em Houghton group \hn} is the subgroup of \sym\ consisting of
those permutations that are eventually translations (of each of the
rays). In other words, the permutation $\sigma$ of the set $R_n$ is in
\hn\ if there exist integers $N\ge 0$ and $t_i=t_i(\sigma)$ for
$i\in\Z_n$ such that for all $k \geq N$, $(i,k)\sigma = (i, k+t_i)$;
throughout  we will use right actions.

Note that necessarily the sum of the translations $t_i$ must be zero
because the permutation needs of course to be a bijection. This implies
that $\H_1\cong \fin$.

For $i,j\in\Z_n$ with $i\neq j$ let $g_{ij}\in\hn$ be the element which
translates the line obtained by joining rays $i$ and $j$, given by
\begin{align*}
(i,n)g_{ij} &= (i,n-1) \text{ if }n > 1,\\
(i,1)g_{ij} &= (j,1),\\
(j,n)g_{ij} &= (j,n+1) \text{ if } n\ge 1\text{ and}\\
(k,n)g_{ij} &= (k,n)\text{ if } k\notin\{i,j\}.
\end{align*}

We also write $g_i$ instead of $g_{i\,i+1}$. It is easy to see that
$\{g_i\mid i\in \Z_n\}$, as well as $\{g_{ij}\mid i,j\in \Z_n,\,i\neq j\}$,
are generating sets for \hn\ if $n\ge 3$ as we can simply check that the
commutator $[g_0,g_1]=g_0^{-1}g_{1}^{-1}g_0g_{1}$ transposes $(1,1)$
and $(2,1)$. In the special case of $\H_2$, $g_1$ is redundant as $g_1=g_0^{-1}$. Further, an additional generator to $g_0$ is
required to generate the group; we choose $\tau$ which fixes all points except for
transposing $(0,1)$ and $(1,1)$.

It is now clear that the commutator subgroup of \hn\ is given by
$$
\hn'=\left\{\begin{array}{ll}
\alt,& \text{if } n\le 2\\
\fin,& \text{if } n\ge 3
\end{array}\right.
$$
For $n\ge 3$, we thus have a short exact sequence
\begin{equation}\label{ses}
\sesn{\fin}{\null}{\hn}{\pi}{\Z^{n-1}}
\end{equation}
where $\pi(\sigma)=(t_0(\sigma),\ldots,t_{n-2}(\sigma))$ is
the abelianization homomorphism. We note that as the sum of all
the eventual translations must be zero, we have the last translation
is determined by the preceding ones:

\begin{equation}\label{eq_zero_sum}
t_{n-1}(\sigma)=-\sum_{i=0}^{n-2}t_i(\sigma).
\end{equation}

We will use the following facts freely throughout this paper, see
Dixon and Mortimer \cite{dixon} or Scott \cite{scott}.

\begin{lemma}\label{lem_free_facts}
The group \alt\ is simple and equal to the commutator subgroup of
\fin, and $\aut(\alt)=\aut(\fin)=\sym$.
\end{lemma}

\section{Automorphisms of \hn}
Here we determine the automorphism group of \hn. First we establish
that we have to look no further than \sym. We let $N_G(H)$ denote the
normalizer, in $G$, of the subgroup $H$ of $G$.

\begin{prop} \label{prop_injective}
Every automorphism of \hn, $n\ge 1$, is given by conjugation by an
element of \sym, that is to say $\aut(\hn)=N_{\sym}(\hn)$.
\end{prop}

\begin{proof}
From the above, the finitary alternating group \alt\ is the
second derived subgroup of \hn, and hence characteristic in \hn.  So
every automorphism of \hn\ restricts to an automorphism of \alt. Since
$\aut(\alt)=\sym$, this restriction yields a homomorphism
$\aut(\hn)\longrightarrow\sym$ and we need to show that it is
injective with image equal to $N_{\sym}(\hn)$.

In order to see this let $\psi\in\aut(\hn)$ be an automorphism. Compose this with an
inner automorphism (of $\sym$) so that the result is an (injective) homomorphism $\alpha: \hn \to \sym$
whose restriction to \alt\ is trivial. We let $k\in\N$ and consider the
following six consecutive points $a_\ell=(i,k+\ell)$ of $R_n$ for
$\ell\in\{0,1,\ldots ,5\}$.

We denote by $g_i^\alpha$ the image of $g_i$ under $\alpha$, and by
$(x\,y\,z)$ the $3$-cycle of the points $x$, $y$ and $z$. Using the
identities
$$
g_i^{-1} (a_1\,a_2\,a_3)g_i = (a_0\,a_1\,a_2) \text{ and }
g_i^{-1} (a_3\,a_4\,a_5)g_i = (a_2\,a_3\,a_4)
$$
and applying $\alpha$, which is trivial on \alt, we get
$$
(g_i^\alpha)^{-1} (a_1\,a_2\,a_3) g_i^\alpha = (a_0\,a_1\,a_2)
\text{ and }
(g_i^\alpha)^{-1} (a_3\,a_4\,a_5) g_i^\alpha = (a_2\,a_3\,a_4).
$$
Hence, we must have that $g_i^\alpha$ maps $\{a_1,a_2,a_3\}$ to $\{a_0,a_1,a_2\}$, and then also $\{a_3,a_4,a_5\}$ to $\{a_2,a_3,a_4\}$. The conclusion  is that it maps $a_3$ to $a_2$. Applying a similar
argument to all points in the branches $i$ and $i+1$, it follows that
$g_i^\alpha=g_i$, and since $i$ was arbitrary, this
means that $\alpha$ is the identity map.
\end{proof}

With Lemma~\ref{lem_free_facts} in mind we now present the complete
description of $\aut(\hn)$.

\begin{thm}\label{thm_aut}
For $n\ge 2$, the automorphism group $\aut(\hn)$ of the Houghton group
\hn\ is isomorphic to the semidirect product
$\hn\rtimes\mathcal{S}_n$, where $\mathcal{S}_n$ is the symmetric
group that permutes the $n$ rays.
\end{thm}

\begin{proof}
By the proposition, it suffices to prove that every
$\alpha\in\sym$ which normalizes \hn\ must map $(i,k+m)$ to $(j,l+m)$
for some $k,l\ge 1$ and all $m\ge 0$.

To this end, we pick $\alpha\in N_\sym(\hn)$ and $\sigma\in\hn$.  Since
$\sigma^\alpha$($=\alpha^{-1}\sigma\alpha)$ is in \hn\ and maps the
point $x\alpha$ to $(x\sigma)\alpha$, these two points must lie on the
same ray for all but finitely many $x\in R_n$. Similarly, $x$ and
$x\sigma$ lie on the same ray for all but finitely many $x\in R_n$, as
$\sigma\in\hn$. In fact, given a ray, we can choose $\sigma$ so that whenever $x$ lies on that ray, then $x$ and $x \sigma$ are successors on the same ray. Combining these facts, we see that $\alpha$ maps
all but finitely many points of ray $i$ to one and the same ray, say
ray $j$. This defines a homomorphism from $\aut(\hn)$ onto $\mathcal{S}_n$, which is obviously split, since given a permutation of the $n$ rays, it clearly defines an automorphism of $\hn$.


So assume that $\alpha$ is now in the kernel of that map, so it does not permute the rays, and take $\sigma$ a $g_{ji}$ generator of $\hn$, i.e., an infinite cycle inside $\sym$. Then, since conjugating inside $\sym$ preserves cycle types, the element $\sigma^\alpha\in\hn$ is also a single infinite cycle. This means that $\sigma^\alpha$ has nonzero translations in only two rays, and these translations are $1$ and $-1$. For all points in the support of $\sigma$, we have that $\sigma^k$ sends this point into the $i$-th ray for all sufficiently large $k$. Therefore, as $\alpha$ sends almost all points in the $i$-th ray into the $i$-th ray, the same is true for $\alpha^\sigma$. Hence $t_i(\sigma^\alpha)$ is positive, so it must be $t_i(\sigma^\alpha)=1$.
It is quite clear now that $\alpha$ translates by an integer in the ray $i$, sufficiently far out. This finishes the proof since this could be done for any $i$, and hence $\alpha\in\hn$.
\end{proof}

\section{Commensurations of \hn}

First, we recall that a commensuration of a group $G$ is an isomorphism
$A\stackrel\phi\longrightarrow B$, where $A$ and $B$ are subgroups of
finite index in $G$. Two commensurations $\phi$ and $\psi$ of $G$ are
equivalent if there exists a subgroup $A$ of finite index in $G$,
such that the restrictions of $\phi$ and $\psi$ to $A$ are equal.  The
set of all commensurations of $G$ modulo this equivalence relation
forms a group, known as the (abstract) \emph{commensurator of $G$},
and is denoted $\com(G)$. In this section we will determine the
commensurator of \hn.

For a moment, we let $H$ be a subgroup of a group $G$. The \emph{relative
  commensurator of $H$ in $G$} is denoted $\com_G(H)$ and consists of
those $g\in G$ such that $H\cap H^g$ has finite index in both $H$ and
$H^g$.  Similar to the homomorphism from $N_G(H)$ to $\aut(H)$, there
is a homomorphism from $\com_G(H)$ to $\com(H)$; Its kernel consists
of those elements of $G$ that centralize a finite index subgroup of
$H$.

In order to pin down $\com(\hn)$, we first establish that every
commensuration of \hn\ can be seen as conjugation by an element of
\sym, and then characterize  $\com_\sym(\hn)$.

Since a commensuration $\phi$ and its restriction to a subgroup of
finite index in its domain are equivalent, we can restrict our
attention to the following family of subgroups of finite index in \hn,
in order to exhibit $\com(\hn)$. For every integer $p\ge 1$, we define
the subgroup $U_p$ of \hn\ by
\[ U_p=\langle \alt, g_i^p\mid i\in\Z_n\rangle.\]

We collect some useful properties of these subgroup first, where \\
$A\subset_f B$ means that $A$ is a subgroup of finite index in $B$.

\begin{lemma}\label{lem_finite_index}
\begin{enumerate} Let $n\ge 3$.
\item[(i)] For every $p$, the group $U_p$ coincides with $\hn^p$, the
  subgroup generated by all $p^\mathrm{th}$ powers in \hn, and hence
  is characteristic in \hn.\\
\item[(ii)] $\displaystyle U_p'=\left\{\begin{array}{cl}\alt, & p \text{
  even}\\\fin, & p \text{ odd}\end{array}\right.$ and
 $\displaystyle \lvert\hn\colon U_p\rvert =\left\{\begin{array}{cl} 2p^{n-1}, & p \text{
  even}\\ p^{n-1}, & p \text{ odd}\end{array}\right.$.\\
\item[(iii)] For every finite index subgroup $U$ of \hn, there exists a $p\ge
  1$ with $\alt = U_p' \subset U_p \subset_f U \subset_f \hn$.
\end{enumerate}
\end{lemma}

The same is essentially true for the case $n=2$, except that $U_p'$ is always equal to $\alt$ in  this case, with the appropriate change in the index.

\begin{proof}
First we establish (ii). We know that $[g_i,g_j]$ is either trivial,
when $j\notin\{i-1,i+1\}$, or an odd permutation.  So the commutator
identities $[ab,c]=[a,c]^b[b,c]$ and $[a,bc]=[a,c][a,b]^c$ imply the
first part, and the second part follows immediately using the short exact sequence (\ref{ses})
from Section~\ref{sec_bg}.

Part (i) is now an exercise, using that $\alt^p=\alt$.

In order to show (iii), let $U$ be a subgroup of finite index in
\hn. The facts that \alt\ is simple and $U$ contains a normal finite
index subgroup of \hn, imply that $\alt\subset U$. Let $p$ be the
smallest even integer such that $(p\Z)^{n-1}$ is contained in the
image of $U$ in the abelianisation of \hn. It is now clear that $U_p$
is contained in $U$.
\end{proof}

Noting that $\com(\H_1) = \aut(\H_1) = \sym$, we now
characterize the commensurators of Houghton's groups.

\begin{thm}\label{thm_com}
Let $n\ge 2$. Every commensuration of \hn\ normalizes $U_p$ for some
even integer $p$.  The group $N_p=N_\sym(U_p)$ is isomorphic to
$\mathcal{H}_{np}\rtimes (\mathcal{S}_p\wr\mathcal{S}_n)$, and $\com(\hn)$ is the
direct limit of $N_p$ with even $p$ under the natural embeddings
$N_p\longrightarrow N_{pq}$ for $q \in\N$.
\end{thm}

\begin{proof}
Let $\phi\in\com(\hn)$. By Lemma~\ref{lem_finite_index}, we can assume
that $U_p$ is contained in the domain of both $\phi$ and $\phi^{-1}$
for some even $p$. Let $V$ be the image of $U_p$ under $\phi$. Then $V$
has finite index in \hn\ and so contains \alt, by
Lemma~\ref{lem_finite_index}.  However, the set of elements of finite
order in $V$ equals $[V,V]$, whence $[V,V]=\alt$, as \alt\ and
\fin\ are not isomorphic. This means that the restriction of $\phi$ to
\alt\ is an automorphism of \alt, and hence yields a homomorphism

\begin{equation*}
\iota\colon\com(\hn) \longrightarrow \sym.
\end{equation*}
%
That $\iota$ is an injective homomorphism to $\com_\sym(\hn)$ follows from a similar argument to the one
in Proposition~\ref{prop_injective} applied to $g_i\,^p$ and six
points of the form $a_\ell = (i, k+p\ell)$ with $\ell\in\{0,1,\ldots ,
5\}$. Since the centralizer in \sym\ of \alt, and hence of any finite
index subgroup of \hn, is trivial, the natural homomorphism from
$\com_\sym(\hn)$ to $\com(\hn)$ mentioned above is also injective, and
we conclude that $\com(\hn)$ is isomorphic to $\com_\sym(\hn)$, and that $\iota$ is the isomorphism:
\begin{equation*}
\iota\colon\com(\hn) \longrightarrow \com_\sym(\hn).
\end{equation*}

%

%
%
%

From now on, we assume that $\phi\in\com_\sym(\hn)$. In particular, the
action of $\phi$ is given by conjugation, and our hypothesis is that
$U_p^\phi\subset\hn$. Now we can apply the argument from the proof of Theorem~\ref{thm_aut} to $\sigma\in U$ and
$\sigma^\phi$ (instead of $\sigma^\alpha$). Namely, consider the $i^{th}$ ray, and choose a $\sigma=g_{ji}^p \in U_p$. So we have
$t_i(\sigma)=p$, $t_j(\sigma)=-p$ and zero translation elsewhere. Now, except for finitely many points, $\sigma^\phi$ preserves the rays and sends $x \phi$ to $x \sigma \phi$. Thus there is an infinite subset of the $i^{th}$ ray which is sent to the same ray by $\phi$, say ray $k$. The infinite subset should be thought of as a union of congruence classes modulo $p$, except for finitely many points. We claim that no infinite subset of ray $j$ can be mapped by $\phi$ to ray $k$. This is because, if it were, the infinite subset would contain a congruence class modulo $p$ (except for finitely many points) from which we would be able to produce two points, $x,y$, in the support of $\sigma^\phi$ such that all sufficiently large positive powers of $\sigma^\phi$ send $x$ into ray $k$ and all sufficiently large negative powers of $\sigma^\phi$ send $y$ into ray $k$, and this is not possible for an element of $\hn$. This means that $\phi$ maps an infinite subset of ray $i$ onto almost all of ray $k$ (observe that ray $k$ is almost contained in the support of $\sigma^\phi$ so must be almost contained in the image of the unions of rays $i$ and $j$). Now applying similar arguments to $\phi^{-1}$ we get that $\phi$ is a bijection between rays $i$ and $k$ except for finitely many points. In fact, $\phi$ must induce bijections between the congruence classes modulo $p$ (except for finitely many points) inside the two ray systems.

Thus $\phi$ induces a permutation of the ray system. Again, looking at large positive powers we deduce that $t_k(\sigma^\phi) > 0 $ and since the support of $\sigma^\phi$ is almost equal to the union of two rays, we must have that $t_k(\sigma^\phi)=p$, as $\sigma$ and hence $\sigma^\phi$ have exactly $p$ orbits. In particular, we may deduce that $\phi$ normalises $U_p$. So $U_p^\phi = U_p$.

In order to proceed, it will be useful to change the ray system. Specifically, each ray can be split into $p$ rays, preserving the order. This realises $U_p$ as a (normal) subgroup of $\hnp$. We say two of these new rays are equivalent if they came from the same old ray. Thus there are $n$ equivalence classes, each having $p$ elements. The group $U_p$ acts on this new ray system as the subgroup of $\hnp$ consisting of all $\sigma \in \hnp $  such that $t_i(\sigma)=t_j(\sigma)$ whenever the $i^{th}$ and $j^{th}$ rays are equivalent. Note that because we have split the rays, these translation amounts can be arbitrary in $U_p$ (as a subgroup of \hnp) and not just multiples of $p$. In particular, we have that $U_p= U_p^\phi\subset\hnp$ and the previous arguments imply that $\phi$ induces a bijection on the new ray system and sends equivalent rays to equivalent rays (since it is actually permuting the old ray system). Since $\phi$ permutes the rays, but must preserve equivalence classes, we get a homomorphism from $N_\sym(U_p)$ to the subgroup of $\mathcal{S}_{np}$ which preserves the equivalence classes - this is easily seen to be $(\mathcal{S}_p\wr\mathcal{S}_n)$ and the above homomorphism is split.

As in Theorem~\ref{thm_aut} we now conclude that the kernel of this homomorphism is $\hnp$ and hence we get that $N_p=N_\sym(U_p)$ is
$\mathcal{H}_{np}\rtimes (\mathcal{S}_p\wr\mathcal{S}_n)$.
\end{proof}

We note that $\com(\hn)$ is not finitely generated, for if it were, it
would lie in some maximal $N_p$.

\section{Metric estimates for \hn}
In this section we will give sharp estimates for the word length of
elements of Houghton's groups. This makes no sense for $\h_1$ which is
not finitely generated. As mentioned in the introduction, the metric
in $\h_2$ was described by Lehnert \cite{jorgdead}. In order to deal
with \hn\ for $n\ge 3$, we introduce the following measure of
complexity of an element.

Given $\sigma\in\hn$, we define $p_i(\sigma)$, for $i\in \Z_n$, to be
the largest integer such that $(i,p_i(\sigma))\sigma\neq (i,p_i(\sigma)
+t_i(\sigma))$.  Note that if $t_i(\sigma) < 0$, then necessarily
$p_i(\sigma)\ge |t_i(\sigma)|$, as the first element in each ray is
numbered $1$.

The \emph{complexity} of $\sigma\in\hn$ is the natural number
$P(\sigma)$, defined by
$$
P(\sigma) = \sum_{i\in\Z_n} p_i(\sigma).
$$
And the \emph{translation amount} of $\sigma$ is
$$
T(\sigma) = \frac12\sum_{i\in\Z_n} |t_i(\sigma)|.
$$

The above remark combined with (\ref{eq_zero_sum}) immediately implies
$P(\sigma)\ge T(\sigma)$. It is easy to see that an element with
complexity zero is trivial, and only the generators $g_{ij}$ have
complexity one.

\begin{thm}
Let $n\ge 3$ and $\sigma\in\hn$, with complexity $P=P(\sigma)\ge 2$. Then
the word length $|\sigma|$ of $\sigma$ with respect to any finite
generating set satisfies $$P/C \leq |\sigma| \leq KP\log P,$$ where
the constants $C$ and $K$ only depend on the choice of generating set.
\end{thm}

\begin{proof}
Since the word length with respect to two different finite generating
sets differs only by a multiplicative constant, we can and will choose
$\{\,g_{ij}\mid i,j\in\Z_n,\,i\neq j\}$ as generating set to work
with, and show that the statement holds with $C=1$ and $K=7$.

The lower bound is established by examining how multiplication by a
generator can change the complexity.  Suppose $\sigma$ has complexity
$P$ and consider $\sigma g_{ij}$.  It is not difficult to see that

\begin{equation}\label{eq_p_k}
p_k(\sigma g_{ij}) = \left\{\begin{array}{cl}
p_k(\sigma) + 1, & \text{ if } k=i \text{ and } (i,p_i(\sigma)+1)\sigma = (i,1)\\
p_k(\sigma) - 1, & \text{ if } k=j,\ (j,p_j(\sigma)+1)\sigma = (j,1) \text{ and }\\
 & (j,p_j(\sigma))\sigma = (i,1)\\
p_k(\sigma), & \text{ otherwise}
\end{array}
\right.
\end{equation}
where the first two cases are mutually exclusive, as $i\neq j$.  Thus
$|P(\sigma g_{ij})- P(\sigma)|\le 1$, which establishes the lower
bound.

The upper bound is obtained as follows. Suppose $\sigma\in\hn$ has
complexity $P$. First we show by induction on $T=T(\sigma)$ that there
is a word $\rho$ of length at most $T\le P$ such that the complexity
of $\sigma\rho$ is $\bar{P}$ with $\bar{P}\leq P$ and $T(\sigma\rho) =
0$. The case $T=0$ is trivial. If $T > 0$, then there are $i,j\in\Z_n$
with $t_i(\sigma) > 0$ and $t_j(\sigma) < 0$. So $T(\sigma g_{ij}) =
T-1$.  Moreover, $P(\sigma g_{ij})\le P$, because the first case of
(\ref{eq_p_k}) is excluded, as it implies that $t_i(\sigma) =
-p_i(\sigma)\leq 0$, contrary to our assumption. This completes the
induction step.

We are now in the situation that $\sigma\rho\in\fin$ and loosely
speaking we proceed as follows.
\begin{enumerate}
\item We push all irregularities into ray $0$, i.e. multiply by $\prod
  g_{i0}\,^{p_i(\sigma\rho)}$.
\item We push all points back into the ray to which they belong,
  except for those from ray $0$ which we mix into ray $1$, say.
\item We push out of ray $1$ separating the points belonging to rays
  $0$ and $1$ into ray $0$ and any other ray, say ray $2$,
  respectively.
\item We push the points belonging to ray $1$ back from ray $2$ into it.
\end{enumerate}

These four steps can be achieved by multiplying by an element $\mu$
of length at most $4\bar{P}$, such that $\sigma\rho\mu$ is an element
which, for each $i$, permutes an initial segment $I_i$ of ray $i$.
Notice that $\sigma \rho \mu$ is now an element of $\hn$ which maps each ray to itself, and hence
$t_i(\sigma \rho \mu)=0$ for all $i$.

It is clear that $\mu$ can be chosen so that the total length of the moved intervals, $\sum |I_i| \le
\bar{P}$. Finally, we sort each of these intervals using a recursive
procedure, modeled on standard merge sort.

In order to sort the interval $I=I_2$ say, we push each of its points
out of ray $2$ and into either ray $0$ if it belongs to the lower
half, or to ray $1$ if it belongs to the upper half of $I$. If each of
the two halves occurs in the correct order, then we only have to push
them back into ray $2$ and are done, having used $2|I|$ generators.
If the two halves are not yet sorted, then we use the same ``separate
the upper and lower halves'' approach on each of them recursively in
order to sort them. In total this takes at most $2|I|\log_2 |I|$
steps.

Altogether we have used at most
\[
P + 4\bar{P} + 2\sum_{i\in \Z_n} |I_i|\log_2 |I_i| \le 7P\log_2P
\]
generators to represent the inverse of $\sigma$; we used the
hypothesis $P\ge 2$ in the last inequality.
\end{proof}

We note that because there are many permutations, the fraction of
elements which are close to the lower bound goes to zero in much the
same way as shown for Thompson's group $V$ by Birget \cite{birget} and
its generalization $nV$ by Burillo and Cleary \cite{2vm}.

\begin{lemma}
Let $n\geq 3$. For $\hn$ take the generating set $g_1,\ldots, g_{n-1}$ with $n-1$ elements. Consider the following sets:
\begin{itemize}
\item $B_k$ is the ball of radius $k$,
\item $C_k$ is the set of elements in $\hn$ which have complexity $P=k$,
\item $D_k\subset C_k$ is the set of elements of $C_k$ which have word length at most $k\log_{2n-2}k$
\end{itemize}
Then, we have that
$$
\lim_{k\to\infty}\frac{|D_k|}{|C_k|}=0
$$
\end{lemma}

An element of complexity $P$, according to the metric estimates proved above, has word length between $P$ and $P\log P$.  What this lemma means is that most elements with complexity $P$ will have word length closer to $P\log P$ than to $P$.

\begin{proof}
Observe that
$$
\frac{|D_k|}{|B_{k\log_{2n-2}k}|}\leq 1
$$
because it is a subset. Now, introduce the $C_k$ as
$$
\frac{|D_k|}{|C_k|}\frac{|C_k|}{|B_{k\log_{2n-2}k}|}
$$
and the proof will be complete if we show that
$$
\lim_{k\to\infty}\frac{|C_k|}{|B_{k\log_{2n-2}k}|}=\infty.
$$
In $C_k$ there are at least $(nk-2)!$ elements. This is because we can take a transposition involving a point at distance $k$ down one of the rays, with another point. Since this already ensures $P=k$, we are free to choose any permutation of the other $nk-2$ points which are at one of the first $k$ positions in each ray. And inside $|B_{k\log_{2n-2}k}|$, counting grossly according to the number of generators, there are at most
$(2n-2)^{k\log_{2n-2}k}=k^k$ elements. Now the limit becomes
$$
\lim_{k\to\infty}\frac{(nk-2)!}{k^k}
$$
which is easily seen that it approaches infinity using Stirling's formula and the fact that $n\geq 3$.
\end{proof}

Consequentially, these estimates give an easy way to see that the group has exponential growth.   We note that exponential growth also follows easily from the fact that $g_{01}$ and $g_{02}$ generate a free subsemigroup.

\begin{prop}
  Let $n\geq3$. Then $\hn$ has exponential growth.
\end{prop}
\begin{proof}
  Consider a finitary permutation of complexity $P$, and observe that there are at least $P!$ of those. By the metric estimate, its word length is at most $KP\log P$. Using again as in the previous lemma the notation $B_k$ for a ball, we will have that the group has exponential growth if
  $$
  \lim_{k\to\infty}\frac{\log|B_k|}k>0.
  $$
  In our case, this amounts to
  $$
  \lim_{P\to\infty}\frac{\log|B_{KP\log P}|}{KP\log P}\geq\lim_{P\to\infty}\frac{\log P!}{KP\log P}=\frac 1K.
  $$
\end{proof}

\section{Subgroup embeddings}

We note that each $\hn$ is a subgroup of $\H_m$ for $n < m$ and that our
estimates together with work of Lehnert are enough to give at least
quadratic distortion for some of these embeddings.

\begin{thm}
The group $\H_2$ is at least quadratically distorted in $\H_m$ for $m
\geq 3$.
\end{thm}

\begin{proof}
We consider the element $\sigma_n$ of $\h_2$ which has $T(\sigma_n)=0$
and transposes $(0,k)$ and $(1,k)$ for all $k \leq n$. Then $\sigma_n$
corresponds to the word $g_n$ defined in Theorem 8 of \cite{jorgdead},
where it is shown to have length of the order of $n^2$ with respect to
the generators of $\h_2$ in Lemma 10 there, which are exactly the
generators for $\h_2$ given in the introduction. One easily checks
that $\sigma_n = g_{02}\,^n g_{1 2}\,^n g_{0 2}\,^{-n} g_{12}\,^{-n}$
in $\h_3$. Thus a family of words of quadratically growing length in
$\h_2$ has linearly growing length in $\h_3$, which proves the theorem.
\end{proof}

A natural, but seemingly difficult, question is whether $\hn$ is
distorted in $\h_m$ for $3\le n < m$.  Another question, which also seems difficult, is to ask whether $\hn$ is distorted in Thompson's group $V$, under the embeddings
 mentioned in the introduction, \cite{claasv}.

\section{Some quasi-isometries of \hn}

Commensurations give rise to quasi-isometries and are often a rich
source of examples of quasi-isometries.  Here we show that the natural
map from the commensurator of \hn\ to the quasi-isometry group of \hn,
which we denote by $\qi(\hn)$, is an injection.  That is, we show that
each commensuration is not within a bounded distance of the identity. That this is an injection also follows from the more general argument of Whyte which appears as Proposition 7.5 in Farb-Mosher \cite{farbmosher}.

\begin{thm}
The natural homomorphism from $\com(\hn)$ to $\qi(\hn)$ is an embedding for $n \geq 2$.
\end{thm}

\begin{proof}
We will show that for each non-trivial $\phi\in\com(\hn)$ and every
$N\in\N$ we can find a $\sigma \in \hn$ such that $d(\sigma,\sigma^\phi)
\ge N$, so none of the non-trivial images are within a bounded distance of the identity.
By Theorem~\ref{thm_com}, we can and will view $\phi$ as a
non-trivial element of $N_p\subset\sym$ for some even $p$.

If  $\phi$ eventually translates a ray $i$ non-trivially to a possibly different ray $j$, then we let
$\sigma=((i,N)\,(i,N+1))$, a transposition in the translated ray.  The
image of $\sigma$ under conjugation by $\phi$ is the transposition
$((j,N+t),(j,N+t'+1))$, and the distance $d(\sigma,\sigma^\phi)$
is the length of $\sigma^{-1}\sigma^\phi$, which is at least $N$
since it moves at least one point at distance $N$ down one of the rays.

If $\phi$ does not eventually translate a ray but
eventually non-trivially permutes ray $i$ with another ray $j$, then we can show boundedness away from the identity by
 taking  $\sigma=((j,N)\,(j,N+1))$.  The point $(i,N)$ is fixed by $\sigma$ but
 is moved to $(i, N+1)$ under $\sigma^\phi$ ensuring that the length of  $\sigma^{-1}\sigma^\phi$ is at least $N$.

Finally, if $\phi$ does not have the preceding two properties, then
$\phi$ is a non-trivial finitary permutation. Since Houghton's group is $k$-transitive, for every $k$, we can find a $\sigma \in \hn$ such that
$\phi^\sigma$ has support disjoint from that of $\phi$, and at distance at least $N$ down one of the rays.
Hence $\sigma^{-1} \phi^{-1} \sigma \phi = \sigma^{-1} \sigma^\phi$ has length at least $N$.

\end{proof}

\section{Co-Hopficity}

Houghton's groups are long known to be Hopfian
although they are not residually finite
, see \cite{isolated}.
In this section we will prove
that \hn\ is not co-Hopfian, by exhibiting a map which is injective
but not surjective. The map is the following:
$$
\begin{array}c
f\colon\hn\longrightarrow\hn\\
s\mapsto f(s)
\end{array}
$$
defined by: if $(i,n)s=(j,m)$, then:
$$
(i,2n-1)f(s)=(j,2m-1)\qquad\text{and}\qquad (i,2n)f(s)=(j,2m).
$$

It is straightforward to show that $f$ is a homomorphism.  It is injective, because if $s$ is not the identity with $(i,n)s\ne(i,n)$, then $(i,2n)f(s)\ne(i,2n)$. And clearly the map is
not surjective, because the permutation always sends adjacent points
$(i,2n-1)$, $(i,2n)$ to adjacent points, and a permutation which does
not do this cannot be in the image.

\medskip

In fact, $\hn$ has many proper subgroups isomorphic to the whole group. The following argument was pointed out to us by Peter Kropholler.

One can well-order the ray system by taking a lexicographic order. The group $\hn$ is then the group of all almost order preserving bijections of the well-ordered ray system. It is then clear that the ray system minus a point is order isomorphic to the original ray system, which demonstrates that a point stabiliser is a subgroup isomorphic to $\hn$.

\begin{thm}
Houghton's groups \hn\ are not co-Hopfian.
\end{thm}

\end{document}